\documentclass[12pt,reqno]{amsart}
\usepackage{hyperref}
\usepackage{amsmath,amssymb}
\numberwithin{equation}{section}

\usepackage{caption}
\usepackage{graphicx}
\usepackage{float}
\usepackage{pgf,tikz}
\usepackage{upgreek}
\usepackage{epsfig}  		

\newcommand \reg{\operatorname{reg}}

\newcommand \h{\operatorname{ht}}

\newcommand\oddgirth{\operatorname{odd-girth}}

\newtheorem{theorem}{Theorem}[section]
\newtheorem{definition}[theorem]{Definition}
\newtheorem{lemma}[theorem]{Lemma}

\newtheorem{example}[theorem]{Example}
\newtheorem{obs}[theorem]{Observation}
\newtheorem{question}[theorem]{Question}

\newtheorem{corollary}[theorem]{Corollary}
\newtheorem{setup}[theorem]{Set-up}

\begin{document}
\title[Regularity of powers of edge ideals of very well-covered
graphs]{Linear polynomial for the regularity of powers of edge ideals of very well-covered graphs}

\author{A. V. Jayanthan}
\address{Department of Mathematics, Indian Institute of Technology
Madras, Chennai, INDIA - 600036}
\email{jayanav@iitm.ac.in}
\author{S. Selvaraja}
\address{Department of Mathematics, Indian Institute of Technology
Madras, Chennai, INDIA - 600036}
\email{selva.y2s@gmail.com}


\maketitle

\begin{abstract}
 Let $G$ be a finite simple graph and $I(G)$ denote the corresponding edge ideal.
 In this paper we prove that if $G$ is a very well-covered graph then for all 
 $s \geq 1$
 the regularity of $I(G)^s$ is exactly $2s+\nu(G)-1$, where $\nu(G)$ denotes the induced
 matching number of $G$. 
\end{abstract}

\section{Introduction}

Let $R=K[x_1,\dots,x_n]$ be a polynomial ring over a field $K$ with
the standard grading (i.e., $\deg(x_i)=1$ for $i=1,\ldots,n$). The
Castelnuovo-Mumford regularity (or simply, regularity) of a finitely generated 
non-zero graded
$R$-module $M$, denoted by $\reg(M)$, is defined to be the least
integer $m$ for which we have for every $j$, the $j^{\textrm{th}}$
syzygy of $M$ is generated in degrees $\leq m+j$.  For a homogeneous
ideal $I$ of $R,$ the behavior of $I^{s}$ for $s\geq 2$ is studied
in various contexts.
It was proved by Cutkosky, Herzog and Trung, \cite{CHT},
and independently by Kodiyalam \cite{vijay}, that for a homogeneous
ideal $I$ in a polynomial ring, $\reg(I^s)$ is given by a linear
function for $s \gg 0$, i.e., there exist non-negative integers $a$,
$b$, $s_0$ such that  $$\reg(I^s) = as + b \text{ for all } s \geq
s_0.$$ They also proved that $a \leq \deg(I)$,
where $\deg(I)$ denotes the maximal degree of a minimal generator.
Finding exact values of $b$ and $s_0$ are non-trivial tasks, even
for monomial ideals, (see, for example, \cite{Conca}, \cite{Ha2}).
There have been some attempts on computing the exact form of this
linear function and the stabilization index $s_0$ for several classes
of ideals, see for example \cite{Berle}, \cite{Chardin2}, \cite{EH}, \cite{EU},
\cite{Ha}. In this paper, we obtain the linear polynomial
corresponding to the regularity of powers of edge ideals of very
well-covered graphs.

Let $G$ be a finite simple (no loops, no multiple edges) undirected graph on the vertex set
$V(G)=\{x_1,\ldots,x_n\}$. Let $K[x_1,\ldots,x_n]$ be the polynomial ring in $n$ variables,
where $K$ is a field. Then the ideal $I(G)$ generated by $\{x_ix_j \mid \{x_i,x_j\} \in E(G)\}$
is called the \textit{edge ideal} of $G$. 
For a graph $G$, there exist integers $b$ and 
$s_0$ such that $\reg(I(G)^s) = 2s + b$ for all $s \geq s_0$.
 Our objective in this paper is to 
find $b$ and $s_0$, for certain class of graphs, in terms of
combinatorial invariants of the graph $G$.
Regularity of edge ideals and
their powers have been studied by several authors and bounds on regularity have been
computed, (see, for example,
\cite{banerjee1}, \cite{ABS}, \cite{banerjee}, \cite{selvi_ha}, \cite{hhz},
\cite{jayanthan}, \cite{moghimian_seyed_yassemi}, \cite{nevo_peeva}, \cite{NSY17}, \cite{selva}).

In \cite{GV05}, Gitler and Valencia proved that if $G$ is a well-covered graph without isolated vertices, then
$\h(I(G)) \geq \frac{|V(G)|}{2}$. In this paper, we consider the class of graphs for which
the above inequality is an equality, namely \textit{very well-covered
graphs}.
The Cohen-Macaulayness, regularity and projective dimension of
very well-covered graphs have already been looked into by several
authors, \cite{crupi}, \cite{KTY}, \cite{mohammad}, \cite{NSY17}, \cite{Fakhari}.
Since the class of very well-covered graphs
contains unmixed bipartite graphs, whiskered graphs and grafted graphs
(see \cite{crupi}, \cite{Faridi05}), it is
interesting in the algebraic sense as well.

The regularity of powers of edge ideals of 
unmixed bipartite graphs have been studied by Jayanthan et al., \cite{jayanthan}.
They showed that if $G$ is an unmixed bipartite graphs, then $\reg(I(G)^s)=2s+\nu(G)-1$ for
all $s \geq 1$.
Mahmoudi et al., \cite{mohammad}, showed that for a very well-covered graph $G$, 
$\reg(I(G))=\nu(G)+1$, where $\nu(G)$ denotes the induced matching number of $G$. 
Since unmixed bipartite graphs are very well-covered
graphs,  it is natural to ask if the same result generalizes to 
very well-covered graphs. Recently, Norouzi et al. showed that if $G$
is a very well-covered graph with $\oddgirth(G) \geq 2k+1$, then $\reg(I(G)^s)=2s+\nu(G)-1$, 
for $1 \leq s \leq k-2$, \cite{NSY17}.

The main result of this paper is the following:
\begin{theorem} (Theorem \ref{main})
 Let $G$ be a very well-covered graph. Then
  for all $s \geq 1$, 
  \[\reg(I(G)^s)=2s+\nu(G)-1.\]
\end{theorem}
Therefore, for this class of graphs, we have $b=\nu(G)-1$ and $s_0=1$.
As an immediate consequence, we get that the above equality holds for
an unmixed bipartite graphs and whiskered graphs.

Our paper is organized as follows. In Section \ref{Preliminaries}, we
collect the necessary notation, terminology and some results that are
used in the rest of the paper. The main tool in obtaining
$\reg(I(G)^{s+1})$ is a result of Banerjee which gives an upper bound
on $\reg(I(G)^{s+1})$ in terms of $\reg((I(G)^{s+1} : M))$ and
$\reg(I(G)^s)$, where $M$ is minimal generator of $I(G)^s$. In Section \ref{squarefree}, we prove that the 
regularity of $(I(G)^{s+1}:M)$ is bounded above by $\nu(G) + 1$, when $(I(G)^{s+1}:M)$ is squarefree.
We study the case when $(I(G)^{s+1}:M)$ has square monomial generators
in Section \ref{main_result} and show that in this case also,
the regularity is bounded above by $\nu(G) + 1$. 
Using these upper bounds, we prove our main result. 

\section{Preliminaries}\label{Preliminaries}
Throughout this paper, $G$ denotes a finite simple graph without
isolated vertices. 
For a graph $G$, $V(G)$ and $E(G)$ denote the set of all
vertices and the set of all edges of $G$ respectively.
A subgraph $H \subseteq G$  is called \textit{induced} if for $u, v
\in V(H)$, $\{u,v\} \in
E(H)$ if and only if $\{u,v\} \in E(G)$.
For $\{u_1,\ldots,u_r\}  \subseteq V(G)$, let $N_G(u_1,\ldots,u_r) = \{v \in V (G)\mid \{u_i, v\} \in E(G)~ 
\text{for some $1 \leq i \leq r$}\}$
and $N_G[u_1,\ldots,u_r]= N_G(u_1,\ldots,u_r) \cup \{u_1,\ldots,u_r\}$. 
For $U \subseteq V(G)$, define $G \setminus U$ to
be the induced subgraph of $G$ on the vertex set $V(G) \setminus U$.

A \textit{matching} in a graph $G$ is a collection of pairwise
disjoint edges. A matching $M$ of a graph $G$ is called an
\textit{induced matching} if no two edges of $M$ are joined by an edge
of $G$.  The largest size of an induced matching in $G$ is called its
\textit{induced matching number} and denoted by $\nu(G)$. A subset $X$
of $V(G)$ is called an \textit{independent set} if $\{x, y\} \notin
E(G)$ for $x, y \in X$. An independent set is said to be a
\textit{maximal independent set} if it is maximal, with respect to
inclusion, among the independent sets.

A subset $M \subseteq V(G)$ is a \textit{vertex cover}  of
$G$ if for each $e \in E(G)$, $e\cap M \neq \emptyset$.  If $M$ is minimal
with respect to inclusion, then $M$ is called a \textit{minimal vertex
cover} of $G$.  A graph $G$ is called \textit{unmixed} (also called
\textit{well-covered}) if all minimal vertex covers of $G$ have the same
number of elements. 

A graph $G$ is called \textit{very well-covered} if it is unmixed without isolated vertices and 
with $\h(I(G))=\frac{|V(G)|}{2}$. The following is a useful result on
very well-covered graphs that allow us to assume certain order on
their vertices and edges.

\begin{lemma}\cite[Corollary 3.2]{gitler}\label{very_well_pro} Let $G$ be a very well-covered graph with $2h$ vertices. Then 
there is a relabeling of vertices $V(G)=\{x_1,\ldots,x_h , y_1,\ldots,y_h\}$ such that the following two conditions hold:
\begin{enumerate}
 \item $X = \{x _1,\ldots,x_h\}$ is a minimal vertex cover of $G$ and $Y = \{y_1,\ldots,y_h\}$ is a maximal independent set of $G$;
 \item For all $1 \leq i \leq h$, $\{x_i,y_i\} \in E(G)$.
\end{enumerate}
\end{lemma}

The concept of \textit{even-connectedness} was introduced by
Banerjee in \cite{banerjee}. This has emerged as a fine tool in the
inductive process of computation of asymptotic regularity.
We recall the definition and some of its important
properties from \cite{banerjee}.

\begin{definition}\label{even_connected} Let $G$ be a graph. Two vertices $u$ and $v$ ($u$ may be same as $v$) are said to be even-connected with respect to an $s$-fold products 
$e_1\cdots e_s$, where $e_i$'s are edges of $G$, not necessarily distinct, if there is a path $p_0p_1\cdots p_{2k+1}$, $k\geq 1$ in $G$ such that:
\begin{enumerate}
 \item $p_0=u,p_{2k+1}=v.$
 \item For all $0 \leq \ell \leq k-1,$ $p_{2\ell+1}p_{2\ell+2}=e_i$ for some $i$.
 \item For all $i$, $ \mid\{\ell \geq 0 \mid p_{2\ell+1}p_{2\ell+2}=e_i \}\mid
   ~ \leq  ~ \mid \{j \mid e_j=e_i\} \mid$.
 \item For all $0 \leq r \leq 2k$, $p_rp_{r+1}$ is an edge in $G$.
\end{enumerate}
\end{definition}

The next theorem describes the minimal generators of the ideal $(I(G)^{s+1}:M)$, where $M$ is minimal generator of $I(G)^s$ for $s\geq 1$.
\begin{theorem}\label{even_connec_equivalent}\cite[Theorem 6.1 and Theorem 6.7]{banerjee} Let $G$ be a graph with edge ideal
$I = I(G)$, and let $s \geq 1$ be an integer. Let $M$ be a minimal generator of $I^s$.
Then $(I^{s+1} : M)$ is minimally generated by monomials of degree 2, and $uv$ ($u$ and $v$ may
be the same) is a minimal generator of $(I^{s+1} : M )$ if and only if either $\{u, v\} \in E(G) $ or $u$ and $v$ are even-connected with respect to $M$.
 \end{theorem}
%
%
Polarization is a process to obtain a squarefree monomial ideal from a
given monomial ideal. 
\begin{definition}\label{pol_def}
Let $M=x_1^{a_1}\cdots x_n^{a_n}$ be a monomial in  $R=K[x_1,\dots,x_n]$. Then we define the squarefree monomial $P(M)$ ({\it polarization} of $M$) as 
$$P(M)=x_{11}\cdots x_{1a_1}x_{21}\cdots x_{2a_2}\cdots x_{n1}\cdots x_{na_n}$$ in the polynomial ring $S=K[x_{ij} \mid 1\leq i\leq n,1\leq j\leq a_i]$.
If $I=(M_1,\dots,M_q)$ is an ideal in $R$, then the polarization of
$I$, denoted by $\widetilde{I}$, is defined as $\widetilde{I}=(P(M_1),\dots,P(M_q))$.
\end{definition}
For various properties of polarization, we refer the reader to \cite{Herzog'sBook}.
In this paper, we repeatedly use one of the important properties of the
polarization, namely:
\begin{corollary} \cite[Corollary 1.6.3(a)]{Herzog'sBook}\label{pol_reg} Let $I$ be a monomial ideal in $K[x_1, \ldots, x_n].$ Then
$$\reg(I)=\reg(\widetilde{I}).$$
\end{corollary}

\section{Bounding the regularity: The squarefree monomial case}\label{squarefree}
We obtain the asymptotic expression for the regularity by using
induction and \cite[Theorem 5.2]{banerjee} which says that
$\reg(I(G)^{s+1}) \leq \underset{M}{\max}\{\reg(I(G)^{s+1}: M) + 2s,
\reg(I(G)^s)\}$, where $M$ is a minimal monomial generator of
$I(G)^s$. For this purpose, one needs to compute $\reg(I(G)^{s+1} :
M)$.  In this section, we obtain an upper bound for the regularity of
$(I(G)^{s+1} : M)$, when $(I(G)^{s+1} : M)$ is squarefree. We first
fix certain set-up for the class of graphs that we consider throughout
this paper.

\begin{setup}\label{setup_well} Let $G$ be a graph with $2h$ vertices,
none of which are isolated and $V(G)=X \cup Y $,
where $X = \{x_1,\ldots,x_h\}$ is a minimal vertex cover of $G$ and  $Y = \{y_1,\ldots,y_h\}$ is a maximal independent set of $G$
such that $\{x_i,y_i\} \in E(G)$, for all $1 \leq i \leq h$.
\end{setup}

The following result is being used repeatedly in this paper:
\begin{lemma}\cite[Proposition 2.3]{crupi}\label{very_well_char}
Let $G$ be a graph as in Set-up \ref{setup_well}. Then $G$ is a very
well-covered if and only if the following conditions hold:
\begin{enumerate}
 \item if $\{z_i, x_j\}, \{y_j,x_k\} \in E(G)$, then $\{z_i,x_k\} \in E(G)$ for distinct $i, j, k$ and for $z_i \in \{x_i,y_i \}$;
 \item if $\{x_i, y_j\} \in E(G)$, then $\{x_i, x_j\} \notin E(G)$.
\end{enumerate}
\end{lemma}
We make an observation which follows directly follows from the Lemma \ref{very_well_char}.
\begin{obs} 
 If $G$ is a very well-covered graph as in Set-up
 \ref{setup_well}, then for any $1 \leq i \leq h$, 
 $G \setminus N_G[x_i,y_i]$, $G \setminus N_{G}[x_i]$ and $G \setminus \{x_i,y_i\}$ are very
 well-covered.
\end{obs}

We begin by showing that if we start with a very well-covered graph,
then we may make certain relabelling of the vertices with the
hypotheses of Set-up \ref{setup_well} being preserved.

\begin{lemma}\label{relabel_covered} 
Let $G$ be a very
well-covered graph satisfying Set-up
\ref{setup_well}. For an $i \in \{1, \ldots, h\}$, let 
$N_G(x_i) \setminus X= \{y_{i_1},\ldots,y_{i_t}\}$,
for some $1 \leq i_1,\ldots,i_t \leq h$. Let $$\displaystyle{X' =
  \left\{x_j' \mid  j \in \{1, \ldots, h\}, x_j' =
  \left\{
	\begin{array}{ll}
	  y_j & \text{ if } j \in \{i_1, \ldots, i_t\} \\
	  x_j & \text{ otherwise } 
	\end{array} \right.\right\}}$$
and $Y' = V(G) \setminus X'$.
Let $G_1$ denote the graph with the above
relabelling. Then $G_1$ with $V(G_1) = X' \cup Y'$ satisfies Set-up
\ref{setup_well} and the properties $(1)$ and $(2)$ of Lemma \ref{very_well_char}.
\end{lemma}
\begin{proof} 
Let $X'=\{a_1,\ldots,a_h\}$ and $Y'=\{b_1,\ldots,b_h\}$.
First we claim that $Y'$ is a maximal
independent set of $G_1$. Suppose not, then there exists an edge $\{b_p,
b_q\} \in E(G_1).$ Since $G$ is a very well-covered, at least one of
$b_p$ and $b_q$ is in $\{x_{j_1}, \ldots, x_{j_t}\}$. 
Suppose $b_p,b_q \in \{x_{j_1}, \ldots, x_{j_t}\}$. 
Let $b_p=x_{j_r}$ and $b_q=x_{j_{r'}}$.
We have $\{x_i, y_{j_r}\}, \{x_i, y_{j_{r'}}\}, \{x_{j_r}, x_{j_{r'}} \} \in E(G)$.
Since $G$ is a very well-covered graph, there is an edge $\{x_i,x_{j_{r'}}\}$ in $G$.
This contradicts Lemma \ref{very_well_char}(2).
Suppose $b_q =
x_{j_r} \in \{x_{j_1}, \ldots, x_{j_t}\}$ for some $1 \leq r \leq t$ and $b_p \notin \{x_{j_1}, \ldots, x_{j_t}\}$. 
Therefore, $\{x_i, y_{j_r}\} \in E(G)$. Since
$G$ is a very well-covered and $\{x_i, y_{j_r}\},
~\{x_{j_r},b_p\} \in E(G)$, we have $\{x_i, b_p\} \in E(G)$, i.e.,
$b_p \in \{x_{j_1}, \ldots, x_{j_t}\}$, which is a contradiction to
the assumption that $b_p \notin \{x_{j_1}, \ldots, x_{j_t}\}$. Therefore $Y'$ is a
maximal independent set in $G_1$. Hence, $X'$ is a minimal vertex
cover of $G_1$.
%
%
%
\end{proof}
In the Lemma \ref{relabel_covered}, we have shown that we may conveniently swap
some of the $x_i$'s and $y_i$'s preserving the hypotheses of Set-up
\ref{setup_well} and the properties $(1)$ and $(2)$ of Lemma
\ref{very_well_char}. However,
arbitrary swapping of $x_i$'s and $y_i$'s may not preserve the
hypotheses of Set-up \ref{setup_well} as can be seen from the
following example.
\begin{example}
Let $G$ be the very well-covered graph on $\{x_1,x_2,x_3, y_1,y_2,y_3\}$ as
given in the figure below.

 \begin{minipage}{\linewidth}
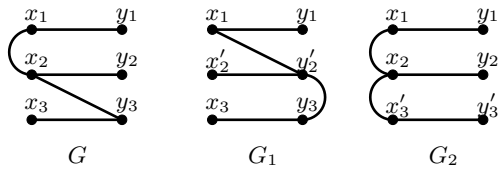
\begin{minipage}{0.5\linewidth}
  \noindent
\begin{figure}[H]
\begin{tikzpicture}[scale=1.2]
\draw [line width=1pt] (0.5,2.5)-- (1.5,2.5);
\draw [line width=1pt] (0.5,2)-- (1.5,2);
\draw [line width=1pt] (0.5,1.5)-- (1.5,1.5);
\draw [shift={(0.5,2.25)},line width=1pt]  plot[domain=1.5707963267948966:4.71238898038469,variable=\t]({1*0.25*cos(\t r)+0*0.25*sin(\t r)},{0*0.25*cos(\t r)+1*0.25*sin(\t r)});
\draw [line width=1pt] (0.5,2)-- (1.5,1.5);
\draw [line width=1pt] (2.5,2.5)-- (3.5,2.5);
\draw [line width=1pt] (2.5,2)-- (3.5,2);
\draw [line width=1pt] (2.5,1.5)-- (3.5,1.5);
\draw [line width=1pt] (4.5,2.5)-- (5.5,2.5);
\draw [line width=1pt] (4.5,2)-- (5.5,2);
\draw [line width=1pt] (4.5,1.5)-- (5.5,1.5);
\draw [line width=1pt] (2.5,2.5)-- (3.5,2);
\draw [shift={(3.5,1.75)},line width=1pt]  plot[domain=-1.5707963267948966:1.5707963267948966,variable=\t]({1*0.25*cos(\t r)+0*0.25*sin(\t r)},{0*0.25*cos(\t r)+1*0.25*sin(\t r)});
\draw [shift={(4.5,2.25)},line width=1pt]  plot[domain=1.5707963267948966:4.71238898038469,variable=\t]({1*0.25*cos(\t r)+0*0.25*sin(\t r)},{0*0.25*cos(\t r)+1*0.25*sin(\t r)});
\draw [shift={(4.5,1.75)},line width=1pt]  plot[domain=1.5707963267948966:4.71238898038469,variable=\t]({1*0.25*cos(\t r)+0*0.25*sin(\t r)},{0*0.25*cos(\t r)+1*0.25*sin(\t r)});
\begin{scriptsize}
\draw [fill=black] (0.5,2.5) circle (1.5pt);
\draw[color=black] (0.5582198376393295,2.6539962216082387) node {$x_1$};
\draw [fill=black] (1.5,2.5) circle (1.5pt);
\draw[color=black] (1.5597727783063593,2.6539962216082387) node {$y_1$};
\draw[color=black] (1.0260985836443652,2.4712310864500213) node {};
\draw [fill=black] (0.5,2) circle (1.5pt);
\draw[color=black] (0.5582198376393295,2.1568750539778874) node {$x_2$};
\draw [fill=black] (1.5,2) circle (1.5pt);
\draw[color=black] (1.5597727783063593,2.1568750539778874) node {$y_2$};
\draw[color=black] (1.0260985836443652,1.97410991881967) node {};
\draw [fill=black] (0.5,1.5) circle (1.5pt);
\draw[color=black] (0.5582198376393295,1.659753886347536) node {$x_3$};
\draw [fill=black] (1.5,1.5) circle (1.5pt);
\draw[color=black] (1.5597727783063593,1.659753886347536) node {$y_3$};
\draw[color=black] (1.0260985836443652,1.4696781457829897) node {};
\draw[color=black] (0.3169698592304829,2.383503821574077) node {};
\draw[color=black] (0.9749243458000646,1.732859940410823) node {};
\draw [fill=black] (2.5,2.5) circle (1.5pt);
\draw[color=black] (2.5613257189733893,2.6539962216082387) node {$x_1$};
\draw [fill=black] (3.5,2.5) circle (1.5pt);
\draw[color=black] (3.5555680542340906,2.6539962216082387) node {$y_1$};
\draw[color=black] (3.029204464978425,2.4712310864500213) node {};
\draw [fill=black] (2.5,2) circle (1.5pt);
\draw[color=black] (2.5613257189733893,2.1568750539778874) node {$x_2'$};
\draw [fill=black] (3.5,2) circle (1.5pt);
\draw[color=black] (3.5555680542340906,2.1568750539778874) node {$y_2'$};
\draw[color=black] (3.029204464978425,1.97410991881967) node {};
\draw [fill=black] (2.5,1.5) circle (1.5pt);
\draw[color=black] (2.5613257189733893,1.659753886347536) node {$x_3$};
\draw [fill=black] (3.5,1.5) circle (1.5pt);
\draw[color=black] (3.5555680542340906,1.659753886347536) node {$y_3$};
\draw[color=black] (3.029204464978425,1.4696781457829897) node {};
\draw [fill=black] (4.5,2.5) circle (1.5pt);
\draw[color=black] (4.557120994901121,2.6539962216082387) node {$x_1$};
\draw [fill=black] (5.5,2.5) circle (1.5pt);
\draw[color=black] (5.558673935568151,2.6539962216082387) node {$y_1$};
\draw[color=black] (5.024999740906157,2.4712310864500213) node {};
\draw [fill=black] (4.5,2) circle (1.5pt);
\draw[color=black] (4.557120994901121,2.1568750539778874) node {$x_2$};
\draw [fill=black] (5.5,2) circle (1.5pt);
\draw[color=black] (5.558673935568151,2.1568750539778874) node {$y_2$};
\draw[color=black] (5.024999740906157,1.97410991881967) node {};
\draw [fill=black] (4.5,1.5) circle (1.5pt);
\draw[color=black] (4.557120994901121,1.659753886347536) node {$x_3'$};
\draw [fill=black] (5.5,1.5) circle (1.5pt);
\draw[color=black] (5.558673935568151,1.659753886347536) node {$y_3'$};
\draw[color=black] (5.024999740906157,1.4696781457829897) node {};
\draw[color=black] (2.970719621727796,2.229981108041174) node {};
\draw[color=black] (3.818749848861923,1.8790720485373968) node {};
\draw[color=black] (4.315871016492274,2.383503821574077) node {};
\draw[color=black] (4.315871016492274,1.8790720485373968) node {};
\draw (0.8,1.3) node[anchor=north west] {$G$};
\draw (2.8,1.3) node[anchor=north west] {$G_1$};
\draw (4.8,1.3) node[anchor=north west] {$G_2$};
\end{scriptsize}
\end{tikzpicture}
\end{figure}
\end{minipage}
\begin{minipage}{0.45\linewidth}
Let $G_1$ be the graph obtained from $G$ by swapping the vertices
$x_2$ and $y_2$, i.e., $V(G_1)=X'\cup Y'$, where
$X'=\{x_1,x_2',x_3\}$ and $Y'=\{y_1,y_2',y_3\}$, where $x_2' = y_2$
and $y_2' = x_2$. 
\end{minipage}
\end{minipage}

We can see that $Y'$ is
not an independent set of $G_1$ and hence $G_1$ does not satisfy the
Set-up \ref{setup_well}.
The third graph, $G_2$, is obtained by swapping the vertices
$x_3$ and $y_3$. Note that $N_G(x_2) \setminus X =
\{y_3\}$. In this case, by taking $X' = \{x_1, x_2, x_3'\}$ and $Y' =
\{y_1, y_2, y_3'\}$, it can be seen that $G_2$ satisfies the Set-up
\ref{setup_well}. 
\end{example}

We now show that adding edges between even-connected vertices
preserves the very well-covered property of a graph, provided there
are no vertices which are even-connected to itself.
\begin{theorem}\label{g=g'}
Let $G$ be a very well-covered graph with $2h$ vertices. If for some
$e \in E(G)$, the ideal $(I(G)^2:e)$ is squarefree, 
then $G'$ is a very well-covered graph, where $G'$ is the graph associated to $(I(G)^2:e)$.
\end{theorem}
\begin{proof} By Lemma \ref{very_well_pro}, there is a relabeling of vertices 
$V(G)=X \cup Y$
such that $G$ satisfies the Set-up \ref{setup_well}, 
where $X=\{x_1,\ldots,x_h\}$ and $Y=\{y_1,\ldots,y_h\}$.
Suppose $e=\{x_n,y_n\}$, for some $1 \leq n \leq h$. 
If $u$ and $v$ are even-connected with respect to $\{x_n, y_n\}$, then
by Lemma \ref{very_well_char}, $\{u, v\} \in E(G)$. Therefore by
Theorem \ref{even_connec_equivalent} $(I(G)^2:e)=I(G)$ so that
$G'$ is a very well-covered graph.

Suppose $e=\{x_n,y_m\}$, for some $1 \leq n,m \leq h$. Note that $G'$ also satisfies Set-up \ref{setup_well}.
We need to show that, $G'$ satisfies Lemma \ref{very_well_char}(1)-(2).
If $\{z_i,x_j\}, \{y_j,x_k\} \in E(G)$, then by Lemma \ref{very_well_char}, $\{z_i,x_k\} \in E(G)$.
Therefore $\{z_i,x_k\} \in E(G')$.
Suppose $\{z_i,x_j\} \in E(G)$ and $\{y_j,x_k\} \in E(G') \setminus E(G)$. Let 
$y_jp_1p_2x_k$ be an even-connection in $G$ with respect to $\{p_1,p_2\}=e$. 
Since $\{z_i, x_j\}$ and $\{y_j, p_1\}$ are in $E(G)$ and
$G$ is very well-covered, $\{z_i,p_1\} \in E(G)$. Hence $z_ip_1p_2x_k$
is an even-connection in $G$ with respect to $e$ so that $\{z_i,x_k\} \in E(G')$.
Similarly we can prove that, if $\{z_i,x_j\} \in E(G') \setminus E(G)$ and $\{y_j,x_k\} \in E(G)$,
then $\{z_i,x_k\} \in E(G')$.
Suppose $\{z_i,x_j\}, \{y_j,x_k\} \in E(G') \setminus E(G)$. Let $z_ip_1p_2x_j$ and
$y_jq_1q_2x_k$ be an even-connection in $G$ with respect to $e=\{p_1,p_2\}=\{q_1,q_2\}$.
If $p_1=q_1$, then there is an even-connection $z_i(p_1=q_1)q_2x_k$
in $G$ with respect to $e$. Suppose $p_1=q_2$. Then
$\{p_2,x_j\} \in E(G)$ and $\{y_j,p_2\} \in E(G)$. This
contradicts the fact that $G$ is a very well-covered graph.
Therefore $\{z_i,x_k\} \in E(G')$.

Now we show that if $\{x_i,y_j\} \in E(G')$, then $\{x_i,x_j\} \notin
E(G')$. Suppose $\{x_i,y_j\} \in E(G)$ and $\{x_i,x_j\} \in E(G')$. Note that $\{x_i,x_j\} \notin E(G)$. 
Let $x_ip_1p_2x_j$ be an even-connection in $G$ with respect to $e=\{p_1,p_2\}$.
Since $\{x_i, y_j\}, \{x_j, p_2\} \in E(G)$,
$\{x_i,p_2\} \in E(G)$.
Then there is an even-connection $x_ip_1p_2x_i$ in $G$ with respect to $e$.
Therefore $x_i^2 \in (I(G)^2:e)$, which is a contradiction. 
Therefore, $\{x_i, x_j\} \notin E(G')$.
Suppose $\{x_i, x_j \}, \{x_i, y_j\} \in E(G')\setminus E(G)$. Therefore, there
exist even-connections, $x_ip_1p_2x_j$ and $x_iq_1q_2y_j$. If
$p_1=q_1$, then there exist edges $\{p_2, x_j\}$ and
$\{p_2, y_j\}$ in $E(G)$ which contradicts the assumption that $G$
is very well-covered. If $p_1 = q_2$, then there
exists an even-connection $x_ip_2p_1x_i$. Therefore, $x_i^2 \in (I(G)^2: e)$ 
which contradicts the assumption that $(I(G)^2:e)$ is a squarefree
monomial ideal. Hence $G'$ is a very well-covered graph.

Suppose $e = \{x_n, x_m\}$, for some $1 \leq n,m \leq h$  and 
$N_G(x_m) \setminus X=\{y_{j_1},\ldots,y_{j_t}\}$. Let $X'$ and $Y'$
be as in Lemma \ref{relabel_covered}.
Let $G_1$ denote the graph with the above
relabelling. By Lemma \ref{relabel_covered},
$V(G_1) = X' \cup Y'$ satisfies Set-up \ref{setup_well} and Lemma
\ref{very_well_char}. Let $e'$ denote the edge $e$ after relabelling.
Then $e' = \{x_n, y_m\}$. Since $I(G_1)$ is obtained from $I(G)$ by
relabelling certain variables, it follows that $(I(G_1)^2 :
e')$ is also a squarefree monomial ideal.  Let $G_1'$ be the
graph associated to $(I(G_1)^2:e')$. By previous case, $G_1'$ is a
very well-covered graph. Since $G_1'$ is also obtained by relabelling
of certain vertices of $G'$, it follows that $G'$ is a very
well-covered graph. 
\end{proof}

The below example shows that if $(I(G)^2 : e)$ is not squarefree, then
the assertion of the Theorem \ref{g=g'} need not necessarily be true.
\begin{example}\label{squarefree_example}
Let \[I=(x_1y_1,x_2y_2,x_3y_3,x_4y_4,x_1x_2,x_1x_4,x_1y_3,x_2y_3,x_2x_4,x_3x_4)
\subset R=K[x_1,\ldots,x_4,y_1,\ldots,y_4]\]
and $G$ be the associated graph. By Lemma \ref{very_well_pro}, $G$ is a very
well-covered graph. It can be seen that $x_4^2, y_3^2 \in (I^2 :
x_1x_2)$ and that \[I(G') = (\widetilde{I^2 : x_1x_2}) =
I+(y_1y_2,y_1x_4,y_1y_3,y_2x_4,y_2y_3,y_3x_4,x_4z_2,y_3z_1)\subseteq
R[z_1,z_2].\]
Since $\mathfrak{p}=(x_1,y_2,y_3,x_4)$ and $\mathfrak{q}=(x_1,x_2,x_3,y_1,y_2,y_3,y_4,z_2)$
are minimal prime ideals of $I(G')$, $G'$ is not a very well-covered graph.
\end{example}

\begin{theorem}\label{g'=sqfree}
 Let $G$ be a graph and $e_1,\ldots,e_s$, $s \geq 1$ be some edges of $G$ which are not
necessarily distinct. Suppose $(I(G)^{s+1}:e_1 \cdots e_s)$ is squarefree ideal. For $1 \leq i \leq s$,
$$(I(G)^{s+1}:e_1 \cdots e_s)=((I(G)^2:e_i)^s: \prod_{i \neq j}e_j).$$
 \end{theorem}
 \begin{proof}
This result has been proved for bipartite graphs in \cite[Lemma
3.7]{banerjee1}. In our case, the assumption that $(I(G)^{s+1} : e_1
\cdots e_2)$ is squarefree implies that there are no odd cycles in the
even-connections. Using this property, one can see that their proof
goes through in our case as well.
 \end{proof}

In \cite[Theorem 4.1]{jayanthan}, it was proved that if $G$ is an
unmixed bipartite graph, then so is $G'$, the graph associated to
$(I(G)^{s+1} : e_1\cdots e_s)$. As a consequence of the above results,
we generalize this to the case of very well-covered graphs $G$ with
$(I(G)^{s+1} : e_1 \cdots e_s)$ squarefree.
 \begin{corollary}\label{g'=verywell} (with hypothesis as in Theorem
   \ref{g'=sqfree}). 
 If $G$ is a very well-covered graph, then  
 so is the graph $G'$ associated
to $(I(G)^{s+1}:e_1 \cdots e_s)$, for every $s$-fold product $e_1 \cdots e_s$ and $s \geq 1$.
 \end{corollary}
\begin{proof} 
We prove that $G'$ is very well-covered by induction on $s$. 
If $s=1$, then the assertion follows from Theorem \ref{g=g'}. 
Assume by induction that for any very well-covered graph $H$ and edges
$f_1, \ldots, f_{s-1}$ with $(I(H)^s : f_1 \cdots f_{s-1})$ a
squarefree monomial ideal, 
the graph associated to $(I(H)^s:f_1 \cdots f_{s-1})$ is very
well-covered.
By Theorem \ref{g'=sqfree} we have 
$$(I(G)^{s+1}:e_1 \cdots e_s)=((I(G)^2:e_i)^s: \prod_{i \neq j}e_j).$$
Note that
$(I(G)^2:e_i)$ is  squarefree monomial ideal.
Let $H$ be the graph associated to $(I(G)^2:e_i)$. 
By the case $s = 1$, $H$ is a very well-covered graph. Therefore, by
induction, the graph associated to
$((I(G)^2:e_i)^s: \prod_{i \neq j}e_j)$ is a very well-covered graph. 
\end{proof}

We finally obtain an upper bound for the regularity in the squarefree case:
 \begin{corollary}\label{reg_sqfree}(with hypothesis as in Theorem
   \ref{g'=sqfree}). Let $G$ be a very well-covered graph. Then 
$$\reg((I(G)^{s+1}:e_1 \cdots e_s)) \leq \nu(G)+1.$$
\end{corollary}
\begin{proof}
 Let $G'$ be the graph associated to $(I(G)^{s+1}:e_1\cdots e_s)$.
  By Corollary \ref{g'=verywell}, $G'$
is a very well-covered graph. Therefore,
\[
\begin{array}{lcll}
     \reg((I(G)^{s+1}:e_1 \cdots e_s))&=& \nu(G')+1& \mbox{ (by \cite[Theorem 4.12]{mohammad})}  \\
	& \leq & \nu(G)+1. & \mbox{ (by  \cite[Proposition 4.4]{jayanthan})}
\end{array}
\]
\end{proof}

\section{Regularity of powers of edge ideals of very well-covered
graphs} \label{main_result}

In this section, we obtain an upper bound for the regularity of
$(I(G)^{s+1} : e_1 \cdots e_s)$ when this is not a squarefree monomial
ideal. Using these results we prove the main theorem, namely, the
asymptotic expression for the regularity of powers of edge ideals of
very well-covered graphs.
\begin{setup}\label{evenconnection_setup}
Let $G$ be a very well-covered graph with $2h$ vertices and
$V(G)=\{x_1,\ldots,x_h,$ $y_1,\ldots,y_h\}$ satisfying Lemma
\ref{very_well_pro}(1 - 2). By Theorem \ref{even_connec_equivalent} and
Definition \ref{pol_def}, $\widetilde{(I(G)^{s+1}:e_1 \cdots e_s)}$ is
a quadratic  squarefree monomial ideal in an appropriate polynomial ring.
Let $G'$ be the graph associated to $\widetilde{(I(G)^{s+1}:e_1 \cdots e_s)}$.
\end{setup}
If $(I(G)^{s+1} : e_1\cdots e_s)$ is not squarefree, then there will be
new vertices along with new edges in $G'$ and hence it 
need not necessarily be a very
well-covered graph, see Example \ref{squarefree_example}.  Our aim in this section is to get an upper bound
for $\reg(I(G'))$.  For this purpose, we need to get more details
about the structure of the graph $G'$. With this aim in mind, in the
next three Lemmas,
we describe some of the edges that are in $G'$ which are possibly not
in $G$.

\begin{lemma}\label{useful} Let the notation be as in Set-up \ref{evenconnection_setup}. 
For $t_i \in \{x_i,y_i \}$, if $t_ix_j$ and $y_jx_k \in I(G')$, then either $t_ix_k \in I(G')$ or $t_iy_j \in I(G')$ 
  for distinct $i, j, k$.
 \end{lemma}
\begin{proof}
Suppose $\{t_i,x_j\} \in E(G)$. 
If $\{y_j,x_k\} \in E(G)$, then by Lemma \ref{very_well_char}, $\{t_i,x_k\} \in E(G)$. Suppose $y_j$ and $x_k$ are 
even-connected with respect to $e_1 \cdots e_s$ in $G$. For some $r' \geq 1$, let
$(y_j=p_0)p_1 \cdots p_{2r'}(p_{2r'+1}=x_k)$
be an even-connection in $G$. 
Since $\{t_i,x_j\}, \{y_j,p_1\} \in E(G)$, by Lemma \ref{very_well_char}, $\{t_i,p_1\} \in E(G)$.
Then there is an even-connection
$t_ip_1 \cdots (p_{2r'+1}=x_k)$ with respect to $e_1 \cdots e_s$ in $G$. 
Therefore, $\{t_i,x_k\}\in E(G')$.
Similarly we can prove that, if $\{t_i,x_j\} \in E(G') \setminus E(G)$ and
$\{y_j,x_k\} \in E(G)$, then $\{t_i,x_k\} \in E(G')$.

Suppose $\{t_i,x_j\},\{y_j,x_k\} \in E(G')\setminus E(G)$.
For some $r_1 \geq 1$ and $r_2 \geq 1$, 
let $$(t_i=q_0)q_1 \cdots q_{2r_1}(q_{2r_1+1}=x_j)\text{ and }
(y_j=s_0)s_1 \cdots s_{2r_2}(s_{2r_2+1}=x_k)$$
be even-connections with respect to $e_1 \cdots e_s$ in $G$.
Suppose $\{s_{2\alpha+1},s_{2\alpha+2}\}$ and $\{q_{2\beta+1},q_{2\beta+2}\}$ 
do not have  common vertices, 
for all $0 \leq \alpha \leq r_2-1$ and $0 \leq \beta \leq r_1-1$. 
Since $\{q_{2r_1},x_j\},\{y_j,s_1\} \in E(G)$, $\{q_{2r_1},s_1\}\in E(G)$.
Then there is an even-connection $(t_i=q_0)q_1 \cdots q_{2r_1}s_1 \cdots s_{2r_2}(s_{2r_2+1}=x_k)$
with respect to $e_1 \cdots e_s$ in $G$.
If for some $0 \leq \alpha \leq r_2-1$, $0 \leq \beta \leq r_1-1$, 
$\{s_{2\alpha+1},s_{2\alpha+2}\}$ and $\{q_{2\beta+1},q_{2\beta+2}\}$ have a common vertex, then
by \cite[Lemma 6.13]{banerjee}, $t_i$ is even-connected to either $y_j$ or $x_k$ with respect to $e_1 \cdots e_s$ in $G$.
Therefore either $\{t_i,y_j\} \in E(G')$ or $\{t_i,x_k\} \in E(G')$.
\end{proof}
\begin{lemma}\label{even_obs} Let the notation be as in Set-up \ref{evenconnection_setup}.  
  Suppose $(u=p_0)p_1 \cdots p_{2k}(p_{2k+1}=v)$ is an even-connection in $G$ with respect to
  $e_1 \cdots e_s$, for some $k \geq 1$. If $\{w,p_i\} \in E(G')$, for some 
  $0 \leq i \leq 2k+1$, then either $\{u,w\} \in E(G')$ or $\{v,w\} \in E(G')$.
 \end{lemma}
\begin{proof}
 If $i=0,2k+1$, then we are done. 
Assume that $i=2j+1$, for some $j \geq 0$.
For some $j \geq 0$, let
$(w=q_0)q_1 \cdots (q_{2j+1}=p_i)$ be an even-connection with respect to $e_1 \cdots e_s$ 
in $G$. If $\{q_{2\alpha+1},q_{2\alpha+2}\}$ and $\{p_{2\beta+1},p_{2 \beta+2}\}$ do not have
a common vertex, for 
all $0 \leq \alpha \leq j-1$, $j \leq \beta \leq k-1$, then 
$(w=q_0)q_1 \cdots (q_{2j+1}=p_i)p_{i+1} \cdots (p_{2k+1}=v)$
is an even-connection with respect to $e_1 \cdots e_s$ in $G$.
Therefore, $wv \in I(G')$.
If $\{q_{2\alpha+1},q_{2\alpha+2}\}$ and $\{p_{2\beta+1},p_{2 \beta+2}\}$ 
have a common vertex,
for some $0 \leq \alpha \leq j-1$, $j \leq \beta \leq k-1$, then
by \cite[Lemma 6.13]{banerjee},
$w$ is  even-connected either to $u$ or to $v$ in $G$. 
Therefore either $wu\in I(G')$ or $wv \in I(G')$.
If $i = 2j+2$, then proof is similar.
%
\end{proof}

In the next lemma, we further obtain more even-connected edges in $G'$.
Let $G$ be a very well-covered graph as in Set-up \ref{setup_well}.
For $u=x_i$ or $y_i$, set $[u] = \{x_i,y_i\}$ 
and $N_G[[u]]=N_G[x_i,y_i]$.
\begin{lemma}\label{even-con-itself} Let the notation be as in Set-up \ref{evenconnection_setup}.
Let $u^2 \in (I(G)^{s+1}:e_1 \cdots e_s)$.
If $a \in (N_{G'}([u]\setminus u)\cap V(G))$ 
and $b \in N_{G}[[u]]$, then $\{a,b\} \in E(G')$.
\end{lemma}
\begin{proof} Since $u^2 \in (I(G)^{s+1}:e_1 \cdots s_s)$, we have an even-connection 
$(p_0=u)p_1 \cdots p_{2k}(p_{2k+1}=u)$ with respect 
to $e_1 \cdots e_s$ in $G$, for some $k \geq 1$. 
Note that, since $\{p_1,u\}, \{p_{2k},u\} \in E(G)$, if $b \in
N_G([u]\setminus u)$, then $p_1 \neq b$ and $p_{2k} \neq b$.
Therefore, $(u=p_0)p_1\cdots p_{2k}b$ and $(u=p_{2k+1})p_{2k}\cdots
p_1b$ are an even-connections in $G$ so that $\{u,b\} \in E(G')$.
Hence, if $a = u$, then $\{a,b\} \in E(G')$.

We now assume that $a \neq u$. Suppose $\{a,[u]\setminus u\} \in E(G)$. 
If either $b=u$ or $b=[u]\setminus u$, then we are done.
If $b \in N_{G}(u)$, then $\{a,b\} \in E(G)$. Suppose $b \in N_{G}([u]\setminus u)$. Since we have 
$\{u,b\} \in E(G')$, by the proof of  Lemma \ref{useful}, $\{a,b\} \in E(G')$. 

Suppose $\{a,[u]\setminus u\} \in E(G')\setminus E(G)$.
For some $t \geq 1$, let $(q_0=a)q_1 \cdots (q_{2t+1}=[u]\setminus u)$
be an even-connection with respect to $e_1 \cdots e_s$ in $G$.
If $\{u,b\}\in E(G)$, then by the proof of Lemma \ref{useful}, $\{a,b\} \in E(G')$.
Suppose $\{[u]\setminus u,b\}\in E(G)$. Note that $u$ is
even-connected to $b$ with even-connections $(u=p_0)p_1 \cdots p_{2k}b$ and 
$(u=p_{2k+1})p_{2k} \cdots p_1b$ with respect to $e_1 \cdots e_s$ in $G$.
Suppose $\{p_{2\lambda+1},p_{2\lambda+2}\} \neq \{q_{2\lambda'+1},q_{2\lambda'+2}\}$, 
for all $0 \leq \lambda \leq k-1$, $0 \leq \lambda' \leq t-1$. Then
either $\{p_{2\lambda+1},p_{2\lambda+2}\} \cap \{q_{2\lambda'+1},q_{2\lambda'+2}\} = 
\emptyset$ or $\{p_{2\lambda+1},p_{2\lambda+2}\} \cap \{q_{2\lambda'+1},
q_{2\lambda'+2}\}$ is a vertex. In either case, it
follows from the proof of Lemma \ref{useful} that $\{a,b\} \in E(G')$. 
Suppose $\{p_{2\lambda+1},p_{2\lambda+2}\}=
\{q_{2\lambda'+1},q_{2\lambda'+2}\}$, for some $0 \leq \lambda \leq k-1$, 
$0 \leq \lambda' \leq t-1$.
Choose the smallest $\lambda'$ such that $\{p_{2\lambda+1},p_{2\lambda+2}\} =
\{q_{2\lambda'+1},q_{2\lambda'+2}\}$ for some $0 \leq \lambda \leq k-1$.
If  $p_{2\lambda+1}=q_{2\lambda'+1}$ and $p_{2\lambda+2}=q_{2\lambda'+2}$, then
there is an even-connection 
$$(a=q_0)q_1 \cdots (q_{2\lambda'+1}=p_{2\lambda+1})(q_{2\lambda'+2}=p_{2\lambda+2})p_{2\lambda+3}
\cdots p_{2k}b$$
with respect to $e_1 \cdots e_s$ in $G$. If  $p_{2\lambda+1}=q_{2\lambda'+2}$ and 
$p_{2\lambda+2}=q_{2\lambda'+1}$, then there is an even-connection 
$$(a=q_0)q_1 \cdots (q_{2\lambda'+1}=p_{2\lambda+2})(q_{2\lambda'+2}=p_{2\lambda+1})p_{2\lambda}
\cdots p_1b$$
with respect to $e_1 \cdots e_s$ in $G$. Therefore $\{a,b\} \in E(G')$.

If $b=u$ and $a \in (N_{G'}([u]\setminus u)\cap V(G))$, 
then proceeding as in the previous case of the proof, one can 
show that $\{a,b\} \in E(G')$.
\end{proof}

To get an upper bound for the regularity of $I(G')$, we need to bound
the regularity of certain induced subgraphs of $G'$. In the next two
lemmas, we understand more closely the structure of some of the
induced subgraphs of $G'$. This, in turn, helps us during the
induction process.

\begin{lemma}\label{tech_lemma} Let the notation be as in  Set-up \ref{evenconnection_setup}.
Let $y \in V(G)$ and $H=G \setminus N_G[y]$. 
If $\{e_1,\ldots,e_s\} \cap E(H) = \{e_{i_1},\ldots,e_{i_t}\}$ and $H'$ is the 
graph associated to $\widetilde{(I(H)^{t+1}:e_{i_1} \cdots e_{i_t})}$, then 
$G' \setminus N_{G'}[y]$ is an induced subgraph of $H'$.
In particular, $\reg(I(G' \setminus N_{G'}[y])) \leq \reg(I(H')).$
\end{lemma}

\begin{proof} 
Let $\{u,v\} \in E(G' \setminus N_{G'}[y])$. 
By Theorem \ref{even_connec_equivalent}, either $\{u,v\} \in E(G)$ or
 $u$ is an even-connected to $v$ in $G$ with respect to $e_1 \cdots e_s$.
If $\{u, v\} \in E(G)$, then $\{u, v\} \in E(H)$.  
Let $(u=p_0)p_1 \cdots (p_{2k+1}=v)$
be an even-connection in $G$ with respect to $e_1 \cdots e_s$ for some
$k \geq 0$. If $p_i \in N_{G'}[y]$, for some $0 \leq i \leq 2k+1$, then by Lemma \ref{even_obs}, $y$
is even-connected either to $u$ or to $v$. This contradicts the
assumption that $\{u,v\} \in G' \setminus N_{G'}[y]$. Therefore, for
each $0 \leq i \leq 2k+1$, $p_i \notin N_{G'}[y]$. Hence $\{u,v\} \in E(H')$, which
proves $G' \setminus N_{G'}[y]$ is a subgraph of $H'$.
If $a,b \in V(G' \setminus
N_{G'}[y])$ is such that $\{a,b\} \in E(H)$,
then $\{a,b\} \in E(G' \setminus N_{G'}[y])$. Hence 
$G' \setminus N_{G'}[y]$ is an induced subgraph of $H$.
The assertion on the regularity follows from \cite[Proposition
4.1.1]{sean_thesis}.  
\end{proof}

It may be noted that, in the above proof, we did not really use the
very well-covered property of $G$. The result holds true for an
arbitrary graph.

\begin{lemma}\label{ind_th_lemma} Let the notation be as in Set-up \ref{evenconnection_setup}. 
Let $u^2 \in (I(G)^{s+1}:e_1 \cdots e_s)$, $t \in 
(N_{G'}([u]\setminus u)) \cap V(G)$ and $H=G \setminus N_{G}[[u]]$. 
Then $G' \setminus N_{G'}[t]$ is 
an induced subgraph of $H'$, where 
$\{e_1,\ldots,e_s\} \cap E(H) =\{e_{i_1},\ldots,e_{i_k}\}$ and
$H'$ is the graph associated to $\widetilde{(I(H)^{k+1}:e_{i_1} \cdots e_{i_k})}$. 
In particular, $\reg(I(G' \setminus N_{G'}[t])) \leq \reg(I(H')).$
\end{lemma}
\begin{proof}
 Let $\{a,b\} \in E(G' \setminus N_{G'}[t])$. 
 By Theorem \ref{even_connec_equivalent}, either $\{a,b\} \in E(G)$ or
 $a$ is an even-connected to $b$ in $G$ with respect to $e_1 \cdots e_s$.
 Suppose $\{a,b\} \in E(G)$. If $\{a, b\} \cap N_G[ [u] ] = \emptyset$, 
 then $\{a,b\} \in E(H)$. If $\{a, b\} \cap N_G[ [u] ] \neq \emptyset$, then
 by Lemma \ref{even-con-itself}, either $\{a,t\} \in E(G')$ or $\{b,t\} \in E(G')$.
 This is a contradiction to $\{a,b\} \in E(G' \setminus N_{G'}[t])$.
 Therefore, if $\{a,b\} \in E(G)$, then $\{a,b\} \cap N_G[ [u] ] =
 \emptyset$ and
 hence $\{a,b\} \in E(H)$.
 
 Suppose
 $\{a,b\} \in E(G') \setminus E(G)$. 
 For $r \geq 1$, let
 $(a=q_0)q_1 \cdots q_{2r}(q_{2r+1}=b)$ be an even-connection in $G$ with
 respect to $e_1 \cdots e_s$. If $q_i \in N_{G}[[u]]$, for some $i$, then by Lemma \ref{even-con-itself},
 $\{t,q_i\} \in E(G')$. Therefore, by Lemma \ref{even_obs}, either
 $\{t,a\} \in E(G')$ or $\{t,b\} \in E(G')$. This is a contradiction
 to the assumption that 
 $\{a,b\} \in E(G' \setminus N_{G'}[t])$.  
 Therefore $q_i \notin N_{G}[[u]]$, for all 
 $0 \leq i \leq 2r+1$ which implies that $a$ is an even-connected to $b$ in
 $H$ with respect to $e_{i_1} \cdots e_{i_k}$, i.e., $\{a,b\} \in E(H')$.
 Hence 
 $G' \setminus N_{G'}[t]$ is an induced subgraph of $H'$.
 The assertion on the regularity follows from \cite[Proposition
4.1.1]{sean_thesis}. 
\end{proof}

Now we prove that the regularity of $I(G')$ is bounded above by
$\nu(G) + 1$.

\begin{theorem}\label{main-techlemma}
Let $G$ be a very well-covered graph and $e_1, \ldots, e_s$ be edges of $G$, for some $s \geq 1$.
Then, 
\[
 \reg((I(G)^{s+1}:e_1 \cdots e_s)) \leq \nu(G)+1.
\]
\end{theorem}
\begin{proof}
For any graph $K$,
let 
$$W_K(e_1 \cdots e_s)=\Big \{ u \in V(K) \mid u \text{ is an even-connected to itself in $K$ with 
respect to $e_1 \cdots e_s$} \Big\}.$$
Let $G'$ be the graph associated to $\widetilde{(I(G)^{s+1}:e_1 \cdots e_s)}$
contained in an  appropriate polynomial ring $R_1$ and $|W_G(e_1 \cdots e_s)|=r$.
We prove the assertion by induction
on $r$. 

If $r=0$, then for any $e_1,\ldots,e_s \in E(G)$, $s \geq 1$, 
$(I(G)^{s+1}:e_1 \cdots e_s)$ is a squarefree monomial ideal. 
Therefore, by Corollary \ref{reg_sqfree}, 
$$\reg((I(G)^{s+1}:e_1 \cdots e_s))\leq \nu(G)+1.$$
By induction, assume that if $L$ is a very well-covered graph with
$|W_L(f_1 \cdots f_s)|<r$ for $f_1,\ldots,f_s \in E(L)$,  then
$\reg((I(L)^{s+1} : f_1 \cdots f_s)) \leq \nu(L) + 1$.

Let $G$ be a very well-covered graph with $|W_G(e_1 \cdots e_s)|
= r$ for $e_1,\ldots,e_s \in E(G)$.  Let $G'$ be the graph
associated to $\widetilde{(I(G)^{s+1}:e_1\cdots e_s)}$, for some $e_1,
\ldots, e_s \in E(G)$.
Set $W_G(e_1\cdots e_s)=\{u_1,\ldots,u_r\}$, $[u_r] \setminus
u_r=u_r', U=(N_{G'}(u_r') \cap V(G))=\{t_1,\ldots,t_l\}$ and
$J=I(G')$. It follows from the exact sequences
\begin{eqnarray*}
  0 & \longrightarrow & \frac{R_1}{(J : t_1)}(-1)
  \overset{\cdot t_1}{\longrightarrow} \frac{R_1}{J} \longrightarrow
  \frac{R_1}{(J , t_1)} \longrightarrow 0; \\
  0 & \longrightarrow & \frac{R_1}{((J, t_1) : t_2)}(-1)
  \overset{\cdot t_2}{\longrightarrow} \frac{R_1}{(J,t_1)} \longrightarrow
  \frac{R_1}{(J, t_1, t_2)} \longrightarrow 0; \\
& & \hspace*{1cm} \vdots\hspace*{5cm}\vdots \hspace*{3cm}\vdots \\
0 & \longrightarrow & \frac{R_1}{((J,
  t_1,\ldots,t_{l-1}):t_l)}(-1) 
\overset{\cdot t_l}{\longrightarrow}
\frac{R_1}{(J, t_1,\ldots,t_{l-1})} \longrightarrow
\frac{R_1}{(J, U)} \longrightarrow 0
\end{eqnarray*}
that
\[
  \reg(R_1/J)  \leq  \max \left\{
	\begin{array}{l}
	  \reg\left(\frac{R_1}{(J :t_1)}\right)+1, ~\reg\left(\frac{R_1}{((J,t_1):
		t_2)}\right)+1,\\
		\ldots \ldots \ldots \\
		\reg\left(\frac{R_1}{(J,t_1,\ldots,t_{l-1}):t_l)}\right) + 1,~ 
		\reg\left(\frac{R_1}{(J,~U)}\right).
\end{array}\right.
\]
We now prove that each of the regularities appearing on the right hand
side of the above inequality is bounded above by $\nu(G)$. 
 
Let $H=G \setminus N_{G}[u_r']$ and $\{e_1, \ldots e_s\} \cap
E(H)=\{e_{i_1}, \ldots, e_{i_k}\}$.
We have
$$
 \begin{array}{lcll}
    \reg(J,U) & = &\reg(I(G' \setminus N_{G'}[u_r'])) & 
    \mbox{ (by \cite[Remark 2.5]{selvi_ha})}\\ 
	& \leq & \reg((I(H)^{k+1}:e_{i_1}\cdots e_{i_k})) & 
	\mbox{ (by Lemma \ref{tech_lemma}})
\end{array}
$$
Since $H$ is a very well-covered graph and $u_r^2 \notin
(I(H)^{k+1}:e_{i_1}\cdots e_{i_k})$,
$|W_H(e_{i_1}\cdots e_{i_k})|<r$. Hence, by induction, we get
$$\reg(J,U) \leq \reg(I(H)^{k+1}:e_{i_1}\cdots e_{i_k}) \leq \nu(H)+1 \leq \nu(G)+1.$$

Let $H = G \setminus N_{G}[[u_r]]$ and $E(H) \cap \{e_1, \ldots, e_s\} =\{e_{i_1},
\ldots, e_{i_\ell}\}$. Since
$H$ is a very well-covered graph and $u_r^2
\notin (I(H)^{\ell+1} : e_{i_1}\cdots e_{i_\ell})$, we have
$$
\begin{array}{lcll}
  \reg(J:t_i) & = & \reg(I(G' \setminus N_{G'}[t_i]))\leq \reg(I(H)^{\ell+1} : e_{i_1}\cdots e_{i_\ell})
  & \text{(By
	Lemma \ref{ind_th_lemma})}\\
	& \leq &\nu(H)+1 & \text{(By induction hypothesis)}\\
	& \leq & \nu(G), &
  \end{array}
$$
where the last inequality follows since $\{f_1, \ldots, f_t, [u_r]\}$ is
an induced matching in $G$ if $\{f_1, \ldots, f_t\}$ is an induced
matching in $H$.

Since $((J,t_1,\ldots,t_{i-1}):t_i)$
corresponds to an induced subgraph of $(J:t_i)$, it follows that
$$\reg \left( \frac{R_1}{((J,t_1,\ldots,t_{i-1}):t_i)}\right)+1 \leq 
\reg \left( \frac{R_1}{(J:t_i)}\right)+1 \leq \nu(G).$$

Therefore, 
$\reg\left( \frac{R_1}{J}\right) \leq \nu(G).$ 
\end{proof}

Now the main theorem can be derived as a consequence of the above
results:
\begin{theorem}\label{main} Let $G$ be a very well-covered graph. Then
  for all $s \geq 1$, 
  \[\reg(I(G)^s)=2s+\nu(G)-1.\]
\end{theorem}
\begin{proof} For any $s \geq 1$, by \cite[Theorem 4.5]{selvi_ha}, we have
$2s+\nu(G)-1 \leq \reg(I(G)^s).$
We need to prove that $\reg(I(G)^s) \leq 2s+\nu(G)-1$, for all $s \geq 1$.
We prove this by induction on $s$. If $s=1$, then the assertion follows from \cite[Theorem 4.12]{mohammad}.
Assume that $s > 1$.
By applying \cite[Theorem 5.2]{banerjee}
and using induction, it is enough to prove that
 for edges $e_1,\ldots,e_s$ of $G$,
 $\reg(I(G)^{s+1}:e_1 \cdots e_s)\leq \nu(G)+1$ for all $s \geq 1$. This follows from
 Theorem \ref{main-techlemma}.
\end{proof}

Since unmixed bipartite graphs are very well-covered graphs, we
obtain 
\begin{corollary}\cite[Corollary 5.1(1)]{jayanthan}\label{reg_unmixed}
If $G$ is an unmixed bipartite graph, then for all
$s \geq 1$, 
\[
 \reg(I(G)^s)=2s+\nu(G)-1.
\]
\end{corollary}
For a graph $G$ on $n$ vertices, let $W(G)$ be the whiskered graph on $2n$ vertices obtained
by adding a pendent vertex (an edge to a new vertex of degree 1) to every vertex of $G$.

Moghimian et al., \cite[Theorem 2.5]{moghimian_seyed_yassemi} proved that $\reg(I(W(C_n)^s)) = 2s + \nu(W(G)) - 1$
for all $s \geq 1$ and Jayanthan et al. \cite[Corollary
5.1(2)]{jayanthan} proved that if $G$ is a bipartite graph, then
$\reg(I(W(G))^s) = 2s + \nu(W(G)) - 1$ for all $s \geq 1$.
Since whiskered graphs are very well-covered graphs, we
obtain asymptotic regularity expression for this class of graphs as
well:

\begin{corollary} If $G$ is a graph, then for all $s \geq 1$,
\[\reg\left(I(W(G))^s \right) = 2s + \nu(W(G)) - 1.\]
\end{corollary}

Next, we study the regularity of powers of
edge ideals of join of very well-covered graphs.

\begin{definition}\label{joingraph}
Let $G_1=(V(G_1),E(G_1))$ and $G_2=(V(G_2),E(G_2))$ be graphs with
disjoint vertex sets. The join of $G_1$ and $G_2$,
denoted by $G_1* G_2$, is the graph on the vertex set $V(G_1) \cup
V(G_2)$ whose edge set
is $E(G_1 *G_2) = E(G_1) \cup E(G_2) \cup \Big\{\{x, y\} \mid x \in V(G_1) \text{ and } y \in V (G_2)
\Big\}.$
\end{definition}

It may be noted that for very well-covered graphs $G_1, \ldots, G_k$,
the product, $G_1* \cdots * G_k$ is not necessarily a very
well-covered graph. However, we obtain the linear expression for
$\reg(I(G_1* \cdots *G_k)^s)$, for all $s \geq 1$.

\begin{corollary} \label{join_scm}
Let $G_1,\ldots,G_k$ be very well-covered graphs with 
$V(G_i) \cap V(G_j)=\emptyset$, for all $1 \leq i \neq j \leq k $. 
Then for all $s \geq 1$,
$$\reg(I(G_1* \cdots * G_k)^s)=2s+\max\{\nu(G_1),\ldots,\nu(G_k)\}-1$$
\end{corollary}
\begin{proof}
By \cite[Theorem 4.5]{selvi_ha} and \cite[Lemma 3.14]{amir}, for all $s \geq 1$
$$2s+\max\{\nu(G_1),\ldots,\nu(G_k)\}-1 \leq \reg(I(G_1 * \cdots * G_k)^s).$$
Let $\mathcal{A}=\{G \mid \reg((I(G)^{s+1}:e_1 \cdots e_s)) \leq \reg(I(G)),
\text{ for any $s$-fold product $e_1 \cdots e_s$, $s \geq 1$}\}$. By Theorem
\ref{main-techlemma}, it follows that $G_1, \ldots, G_k \in \mathcal{A}$.
Hence by \cite[Theorem 4.4]{selva}, we get $G_1*\cdots
* G_k \in \mathcal{A}$. Therefore, it follows from \cite[Theorem 4.5]{selva} and \cite[Lemma
3.14]{amir}, that for all $s \geq 1$,
$\reg(I(G_1* \cdots * G_k)^s) \leq 2s+\max \{\nu(G_1),\ldots,\nu(G_k)\}-1.$
Hence the assertion follows.
\end{proof}

It follows from Theorem \ref{main} that if $G$ is a very well-covered
graph, then $\reg(I(G)^s) = 2s + \nu(G) - 1$ for all $s \geq 2$.
As a natural extension of this result, one tend to think that the same
expression may hold true for well-covered graphs. This is not the
case. For example, let 
$I=(x_1x_2,x_2x_3,x_3x_4,x_4x_5,x_5x_1,x_1x_6,x_6x_9,x_6x_7,x_7x_8)$
and $G$ be the associated graph. It can be easily verified that $G$ is
a  well-covered graph with $\nu(G)=2$, but not a very well-covered graph.
 By \cite[Corollary 3.8 and Theorem 5.3]{ABS}, for $s \geq 1$,
$2s+\nu(G)-1<\reg(I^s)=2s+\nu(G)=2s+2.$
Also, there exist well-covered graphs $G$  such that $\reg(I(G)^s) =
2s + \nu(G) - 1$ for all $s \geq 2$. For example,
if $G=C_5$, then by \cite[Theorem 5.2]{selvi_ha},
$\reg(I(G)^s) < 2s+\nu(G)-1$ for all $s \geq 2$.
Beyarslan et al. raised the question for which classes of
graphs the equality $\reg(I(G)^s) = 2s + \nu(G) - 1$ holds for $s \gg
0$, \cite[Question 5.4]{selvi_ha}. This seems to be a rather tough question to answer. 
Therefore, we would like to ask:
\begin{question}
Characterize well-covered graphs $G$ for which $\reg(I(G)^s) = 2s +
\nu(G) - 1$ for all $s \gg 0$?
\end{question}

\vskip 2mm
\noindent
\textbf{Acknowledgement:} The computational commutative algebra
package Macaulay 2 \cite{M2} was heavily used to compute several
examples.  We also would
like to thank Selvi Beyarslan for going through the manuscript and
making some useful suggestions. The second author would like to thank the National Board for
Higher Mathematics, India for the financial support. We also thank the
referee for carefully reading the manuscript and making several
suggestions that improved the exposition.

\bibliographystyle{abbrv}  
\bibliography{refs_reg}
\end{document}